%% file: Relativistic.II.tex
\documentclass[12pt,oneside,openany,article]{memoir}

\nouppercaseheads

\usepackage[amsthm,thmmarks,hyperref]{ntheorem}

\usepackage[noTeX]{mmap}

\usepackage[english.us]{babel}
\usepackage[leqno]{mathtools}

\usepackage{caption}

\usepackage[scr]{rsfso}
\usepackage{upgreek}

\usepackage[compress]{cite}
\usepackage{hypernat}

\usepackage{fix-cm}
\usepackage[cm]{sfmath}
\usepackage{sansmathaccent}

\usepackage{microtype}

\usepackage{subfig}
\usepackage{wrapfig}
\usepackage{array}

\usepackage{graphicx,xcolor}
\usepackage{eso-pic}

\usepackage{pdfpages}

\usepackage{todonotes}

\usepackage{hyperxmp}
\usepackage{zref, nameref}

\usepackage{xr-hyper}
\usepackage{url}
\usepackage[pdftex,bookmarks,pdfnewwindow,plainpages=false,unicode,pdfencoding=auto, backref=page]{hyperref}
\usepackage{backref}

\usepackage{bookmark}

\usepackage{enumitem}

\usepackage{amssymb} 

\hypersetup{
colorlinks=true,
linkcolor=black,
citecolor=black,
urlcolor=blue,
pdfauthor={Victor Ivrii},
pdftitle={Asymptotics of the ground state energy in the relativistic settings and with self-generated magnetic field},
pdfsubject={Asymptotics of the ground state energy in the relativistic settings and with self-generated magnetic field},
pdfkeywords={Relativistic Schroedinger operator, Self-Generated Magnetic Field, Heavy atoms and Molecules, Thomas Fermi theory, Scott correction term, Microlocal Analysis,  Sharp Spectral Asymptotics},
bookmarksdepth={4}
}
\theoremstyle{plain}
\newtheorem{theorem}{Theorem}[chapter]

\newtheorem{proposition}[theorem]{Proposition}
\newtheorem{corollary}[theorem]{Corollary}

\theoremstyle{definition}

\theoremstyle{remark}
\newtheorem{remark}[theorem]{Remark}

\makeatletter
\newtheoremstyle{plainfoot}%
  {\item[\hskip\labelsep \theorem@headerfont ##1\ ##2\,\footnotemark\theorem@separator]}%
  {\item[\hskip\labelsep \theorem@headerfont ##1\ ##2\ (##3)\, \footnotemark\theorem@separator]}
\makeatother

\theoremstyle{plainfoot}
\theoremseparator{.}
\newtheorem{theorem-foot}[theorem]{Theorem}
\newtheorem{lemma-foot}[theorem]{Lemma}
\newtheorem{proposition-foot}[theorem]{Proposition}
\newtheorem{corollary-foot}[theorem]{Corollary}
\newtheorem{conjecture-foot}[theorem]{Conjecture}
\newtheorem{condition-foot}[theorem]{Condition}

\theoremstyle{plainfoot}
\theoremseparator{.}
\theorembodyfont{}
\newtheorem{definition-foot}[theorem]{Definition}
\newtheorem{Problem-foot}[theorem]{Problem}

\theoremstyle{plainfoot}
\theoremseparator{.}
\theorembodyfont{}
\theoremheaderfont{\slshape}
\newtheorem{remark-foot}[theorem]{Remark}         
\newtheorem{example-foot}[theorem]{Example}
\newtheorem{problem-foot}[theorem]{Problem}

\input{defines}
\newcommand{\RTF}{{\mathsf{RTF}}}

\newcommand{\e}{{\mathsf{e}}}

\numberwithin{equation}{chapter}

\title{Asymptotics of the ground state energy in the relativistic settings and with self-generated magnetic field}

\author{Victor Ivrii}

\begin{document}

\maketitle

\begin{abstract}
The purpose of this paper is to derive sharp asymptotics of  the ground state energy for the heavy atoms and molecules in the relativistic settings, with the self-generated magnetic field, and, in particular, to derive relativistic Scott correction term and also Dirac, Schwinger and relativistic correction terms. Also we will prove that Thomas-Fermi density approximates the actual density of the ground state, which opens the way to estimate the excessive negative and positive charges and the ionization energy.
\end{abstract}

\chapter{Introduction}
\label{sect-1}

Multielectron Hamiltonian is defined by
\begin{gather}
\mathsf{H}=\mathsf{H}_N\coloneqq   \sum_{1\le j\le N} H _{V,x_j}+\sum_{1\le j<k\le N}\frac{\e^2}{|x_j-x_k|}
\label{1-1}\\
\shortintertext{on}
\fH= \bigwedge_{1\le n\le N} \sH, \qquad \sH=\sL^2 (\bR^3, \bC^q)\simeq \sL^2 (\bR^3\times \{1,\ldots,q\},\bC) \\
\shortintertext{with}
H_V =T - \e V(x),
\label{1-3}
\end{gather}
describing $N$ same type particles in the external field with the scalar potential $-V$  and repulsing one another according to the Coulomb law; $\mathsf{e}$ is a charge of the electron, $T$ is an \emph{operator of the kinetic energy}. Unless specifically mentioned, $q=2$.

In the non-relativistic framework this operator is defined as
\begin{phantomequation}\label{1-4}\end{phantomequation}
\begin{align}
&T= \frac{1}{2\mu} (-i\hbar \nabla-\e A)^2,
\tag*{$\textup{(\ref*{1-4})}_{1}$}\label{1-4-2}\\
&T= \frac{1}{2\mu} \bigl((i\nabla -\e A)\cdot \boldupsigma \bigr) ^2
\tag*{$\textup{(\ref*{1-4})}_{2}$}\label{1-4-3}
\end{align}
in the  magnetic (Schr\"odinger) and  (Schr\"odinger-Pauli) settings respectively.

In the relativistic framework this operator is defined as
\begin{phantomequation}\label{1-5}\end{phantomequation}
\begin{align}
&T= \Bigl(c^2 (-i \hbar \nabla-\e A)^2+\mu^2 c^4\Bigr)^{\frac{1}{2}}-\mu^2c^4
\tag*{$\textup{(\ref*{1-5})}_{1}$}\label{1-5-2}\\
&T= \Bigl(c^2  \bigl((-i\hbar \nabla -\e A)\cdot \boldupsigma \bigr) ^2+\mu ^2 c^4\Bigr)^{\frac{1}{2}}-\mu^2c^4
\tag*{$\textup{(\ref*{1-5})}_{2}$}\label{1-5-3}
\end{align}
in the  magnetic (Schr\"odinger) and  (Schr\"odinger-Pauli) settings respectively.

Recall that in non-magnetic settings we have (\ref{rel-1-4}) and (\ref{rel-1-5}) of \cite{ivrii:rela} in the non-relativistic and relativistic settings respectively. 
Here
\begin{gather}
V(x)=\sum _{1\le m\le M}  \frac{Z_m \e}{|x-\y_m|}
\label{1-6}\\
\shortintertext{and}
d=\min _{1\le m<m'\le M}|\y_m-\y_{m'}|>0.
\label{1-7}
\end{gather}
where $Z_m\mathsf{e}>0$ and $\y_m$ are charges and locations of nuclei.

It is well-known that the non-relativistic operator is always semibounded from below. On the other hand, it 
 is also well-known \cite{Herbst, Lieb-Yau} that one particle relativistic non-magnetic operator is semibounded from below if and only if $Z_m \beta \le \frac{2}{\pi}$ for $m=1,\ldots,M$. In this paper we assume a strict condition:
\begin{equation}
Z_m \beta \le \frac{2}{\pi}-\epsilon\qquad \forall m=1,\ldots, M;\qquad \beta\coloneqq \frac{\e^2}{\hbar c}.
\label{1-8}
\end{equation}
In the non-magnetic case  we were interested in $\E\coloneqq \inf\Spec(\mathsf{H})$. In the magnetic case we consider only a self-generated magnetic field, that is we consider
\begin{gather}
\E ^*=\inf_{A\in \sH^1_0} \E(A),
\label{1-9}\\\shortintertext{where}
\E(A)= \inf \Spec (\sfH_{A,V} ) + 
\underbracket{\frac{\e^2}{\alpha \hbar ^2} \int |\nabla \times A|^2\,dx}
\label{1-10},\\
\alpha Z_m \le \kappa^* (2\pi^{-1} -\beta Z_m)^{\frac{3}{2}} \qquad m=1,\ldots, M.
\label{1-11}
\end{gather}
with a unspecified constant $\kappa^*>0$. We also assume that $d\ge CZ^{-1}$.

\begin{remark}\label{rem-1-1}
\begin{enumerate}[label=(\roman*), wide, labelindent=0pt]
\item\label{rem-1-1-i}
In the non-relativistic theory by scaling with respect to the spatial and energy variables we can make $\hbar=\e=\mu=1$ while  $\alpha$ and $Z_m$  are preserved.

\item\label{rem-1-1-ii}
In the relativistic theory by scaling with respect to the spatial and energy variables we can make 
$\hbar=\e=\mu=1$ while $\beta$, $\alpha$ and $Z_m$ are preserved.

\item\label{rem-1-1-iii}
In the one particle case there are additional scalings with respect to the spatial and energy variables, preserving only 
$Z_m\alpha$ and $Z_m\beta$ (but not the $Z_m,\alpha,\beta$).
\end{enumerate}
From now on we assume that such rescaling was done and we are free to use letters $\hbar$, $\mu $ and $c$ for other notations.
\end{remark}

The sharp results in the non-relativistic frameworks, without magnetic field and with self-generated magnetic filed were obtained in Chapters~\ref{book_new-sect-25} and~\ref{book_new-sect-27} of \cite{monsterbook} respectively, and in the relativistic frameworks without magnetic field––in \cite{ivrii:rela}. The transition from the non-relativistic framework to the relativistic one required mainly modifications of the function-analytic arguments in the \emph{singular zone\/}  $\bigcup_m \{x\colon |x-\y_m|\le cZ_m^{-1}\}$, and it was done in many articles, listed in the references, which we heavily use. On the other hand, transition from the non-magnetic case to the case of the self-generated magnetic field requires microlocal semiclassical arguments of Chapter~\ref{book_new-sect-27} of \cite{monsterbook} in the \emph{semiclassical zone\/} $\bigcap_m \{x\colon |x-\y_m|\ge cZ_m^{-1}\}$, which we also heavily rely upon. However relativistic settings require modifications of these arguments, and we are providing most of details when such modifications are needed, and are rather sketchy when no modifications are required.

\pagebreak

\chapter{Local semiclassical trace asymptotics}
\label{sect-2}

\section{Set-up}
\label{sect-2-1}

This section matches to Section~\ref{book_new-sect-27-2} of \cite{monsterbook}. We consider potential $W$ supported in 
$B(\x, r)$ (with $r=\ell(\x)$ the half-distance to the nearest nucleus), and scale it to $B(0,1)$ with $W\asymp 1$.

Recall that the original non-relativistic operator is 
\begin{gather}
\frac{1}{2} ((D- A)\cdot\boldupsigma)^2 -W,
\label{2-1}\\
\intertext{which after rescaling $x\mapsto (x-\x)/r$, $\tau\mapsto \tau /(Zr^{-1})$ becomes}
\frac{1}{2} ((h D- A')\cdot\boldupsigma)^2 -W,\qquad h= Z^{-1/2}r^{-1/2},\qquad A'=Z^{-1/2}r^{1/2} A,
\label{2-2}\\
\shortintertext{while the ``penalty'' becomes}
\frac{r}{\alpha } \int |\nabla \times  A'|^2\,dx= \frac{1}{\kappa h^2 } \int |\nabla \times  A'|^2\,dx
\label{2-3}
\end{gather}
with $\kappa = Z\alpha $ and we assume that $\kappa \le \kappa^*$.

What happens with our relativistic operator? The same scaling transforms 
$\bigl(\beta^{-2} ((D-A)\cdot\boldupsigma)^2 +\beta^{-4}\bigr)^{1/2}-\beta^{-2}$ into
\begin{multline}
rZ^{-1} \bigl(\beta^{-2}( (r^{-1} D-A)\cdot\boldupsigma)^2 +\beta^{-4}\bigr)^{1/2}-rZ^{-1}\beta^{-2} =\\
\bigl(\gamma^{-2} ((hD -A')\cdot\boldupsigma)^2 +\gamma^{-4} \bigr)^{1/2} -\gamma^{-2}
\label{2-4}
\end{multline}
with $\gamma=\beta h^{-1}\le 1$.

Exactly like in Subsubsection \ref{book_new-sect-27-2-1-2} of \cite{monsterbook} we need to start with the functional-analytic arguments.

\section{Functional analytic arguments}
\label{sect-2-2}

\subsection{Estimates}
\label{sect-2-2-1}

\begin{proposition-foot}\label{prop-2-1}\footnotetext{\label{foot-1} Cf. Proposition~\ref{book_new-prop-27-2-1} of \cite{monsterbook}.}
Let  $V\in \sL^{\frac{5}{2}}\cap \sL^4$. Then
\begin{gather}
\E^*\ge -C h^{-3}
\label{2-5}\\
\shortintertext{and either}
\frac {1}{\kappa h^2} \int |\partial A|^2\,dx \le Ch^{-3}
\label{2-6}
\end{gather}
or $\E (A) \ge ch^{-3}$.
\end{proposition-foot}

\begin{proof}
Using  Theorem~\ref{thm-A-1} (magnetic Daubechies inequality rather than magnetic Lieb-Thirring inequality) with 
$\gamma \coloneqq \gamma h$, $V\coloneqq h^{-2}V$, $A\coloneqq h^{-1}A$  and with multiplication of the result by $h^2$, we have 
\begin{multline}
\Tr (H_{A,V}^-) \ge \\
-Ch^{-3}\int \Bigl(V_+^{5/2}+\gamma^3 V_+^4\Bigr)\,dx -Ch^{-2}\Bigl(|\partial A|^2\,dx\Bigr)^{\frac{1}{4}} 
\Bigl(V_+^4\,dx\Bigr)^{\frac{3}{4}} 
\label{2-7}
\end{multline}
(cf. (\ref{book_new-27-2-9}) of \cite{monsterbook}; only the term $\gamma^3 V_+^4$ adds up); then (\ref{book_new-27-2-10}) holds, which completes the proof.
\end{proof}

\begin{proposition-foot}\label{prop-2-2}\footnotetext{\label{foot-2} Cf. Proposition~\ref{book_new-prop-27-2-2} of \cite{monsterbook}.}
Let  $V_+\in  \sL^{\frac{5}{2}}\cap \sL^4$, $\kappa \le ch^{-1} $ and
\begin{equation}
V\le - C^{-1} (1+|x|)^\delta +C.
\label{2-8}
\end{equation}
Then there exists a minimizer $A$.
\end{proposition-foot}

\begin{proof}
Let us consider a minimizing sequence $A_j$. Without any loss of the generality one can assume that $A_j\to A_\infty$ weakly in $\sH^1$ and  in $\sL^6$ and then strongly in $\sL^p_\loc$ with any $p<6$\,\footnote{\label{foot-3} Otherwise we select a converging subsequence.}. Then $A_\infty$ is a minimizer.

Really, due to (\ref{2-6}) and (\ref{2-8}) negative spectra of $H_{A_j,V}$ are discrete and the number of  negative eigenvalues is bounded by $N=N(h)$. Consider   ordered eigenvalues $\lambda_{j,k}$ of $H_{A_j,V}$. Without any loss of the generality one can assume that $\lambda_{j,k}$ have limits $\lambda_{\infty,k}\le 0$ (we go to the subsequence if needed).

We claim that $\lambda_{\infty,k}$ are also eigenvalues and if $\lambda_{\infty,k}=\ldots=\lambda_{\infty, k+r-1}$ then it is eigenvalue of at least multiplicity $r$. 

Indeed, let $u_{j,k}$ be corresponding eigenfunctions, orthonormal in $\sL^2$. 
Then in virtue of $A_j$ being bounded in $\sL^6$ and $V\in \sL^4$ we can estimate
\begin{equation*}
\||D|^{1/2}u_{j,k}\|\le K \|u_{j,k}\|_{6}^{1-\delta}\cdot \|u_{j,k}\|^\sigma\le K\||D|^{1/2}u_{j,k}\|^{1-\delta}\cdot \|u_{j,k}\|^\delta
\end{equation*}
with $\delta>0$ which implies $\||D|^{1/2}u_{j,k}\|\le K$. Also assumption (\ref{2-8}) implies that 
$\| (1+|x|)^{\delta/2 }u_{j,k}\|$ are bounded and therefore without any loss of the generality one can assume that $u_{j,k}$ converge strongly.

Then
\begin{gather}
\lim_{j\to\infty} \Tr (H_{A_j,V}^-) \ge \Tr (H_{A_\infty,V}^-),
\label{2-9}\\\\
\liminf_{j\to\infty}\int |\partial A_j|^2\, dx \ge
\int |\partial A_\infty |^2\, dx
\label{2-10}
\end{gather}
and therefore $\E(A_\infty)\le \E^*$. Then $A_\infty$ is a minimizer and there are equalities in (\ref{2-9})--(\ref{2-10}) and, in particular, there no negative eigenvalues of $H_{A_\infty, V}$  other than $\lambda_{\infty,k}$. 
\end{proof}

\subsection{Properties of the minimizer}
\label{sect-2-2-2} 
Next, we need to study the minimizer\footnote{\label{foot-4} We do not know if it is unique, exactly like in the non-relativistic case; see Remark~\ref{book_new-rem-27-2-3} of \cite{monsterbook}.}.

\begin{proposition-foot}\label{prop-2-3}\footnotetext{\label{foot-5} Cf. Proposition~\ref{book_new-prop-27-2-4} of \cite{monsterbook}. Observe that (\ref{2-11}) is more complicated than (\ref{book_new-27-2-14}) of \cite{monsterbook}.}
Let $A$ be a minimizer. Then
\begin{multline}
\frac{2}{\kappa h^2} \Delta A_j (x)   = \Phi_j\coloneqq  \\
\begin{aligned}
&- \Re \tr \Bigl[\int_0^\infty\upsigma_j ( (hD -A)_x\cdot \boldupsigma) e^{-\lambda S}e(.,.,0)  e^{-\lambda S}\,d\lambda\Bigr]\Bigr|_{x=y}\\
&- \Re \tr \Bigl[\int_0^\infty\upsigma_j  e^{-\lambda S}e(.,.,0)  e^{-\lambda S}\,^t( (hD -A)_y\cdot \boldupsigma)\,d\lambda\Bigr]\Bigr|_{x=y}.
\end{aligned}
\label{2-11}
\end{multline}
where $A=(A_1,A_2,A_3)$, $\boldupsigma=(\upsigma_1, \upsigma_2, \upsigma_3)$ and $e(x,y,\tau)$ is the Schwartz kernel of the spectral projector $\uptheta (- H)$ of $H=H_{A,V}$ and $\tr$ is a matrix trace;
\begin{equation}
S= \gamma^2 (T+ \gamma^{-2}) =\bigl( (\gamma^2(hD-A)\cdot\boldupsigma )^2+1\bigr)^{\frac{1}{2}}
\label{2-12}
\end{equation}
\end{proposition-foot}

\begin{proof}
Consider variation $\updelta A$ of $A$ and variation of $\Tr (H^-)$ where $H^-=H \uptheta (- H)$ is a negative part of $H$. Then, like in  the proof of Proposition~\ref{book_new-prop-27-2-4},
\begin{equation}
\updelta \Tr(H^-)= \Tr \bigl((\updelta H) \uptheta (- H)\bigr) .
\label{2-13}
\end{equation}
But we need to find $\updelta H=\gamma^{-2} \updelta  S$, which is a bit more tricky than in the non-relativistic case. Observe that 
\begin{gather}
\updelta (S^2)=S(\updelta S) +(\updelta S) S;
\label{2-14}\\
\shortintertext{then}
\updelta S= \int_0^\infty e^{-\lambda S}\updelta (S^2)e^{-\lambda S}\,d\lambda
\label{2-15}\\
\shortintertext{while}
\gamma^{-2}S^2=((hD-A)\cdot\boldupsigma)^2+\gamma^{-2}\label{2-16}\\
\intertext{and therefore}
\updelta (\gamma^{-2}S^2)=  - \sum_j \Bigl(\updelta A_j \upsigma_j ( (hD -A)\cdot \boldupsigma) - 
 ( (hD -A)\cdot \boldupsigma) \updelta A_j \upsigma_j\Bigr),
 \label{2-17}
\end{gather}
exactly like in non-relativistic case. 

Therefore $\Tr (\updelta S \theta(-H))$
is equal to the sum of $\int_0^\infty\,d\lambda $ of
\begin{align*}
-&\Tr \Bigl( e^{-\lambda S}\updelta A_j \upsigma_j ( (hD -A)\cdot \boldupsigma) e^{-\lambda S}\uptheta(-H)\Bigr)\\
-&\Tr \Bigl(e^{-\lambda S}( (hD -A)\cdot \boldupsigma) \updelta A_j \upsigma_je^{-\lambda S}\uptheta(-H)\Bigr)=\\
-&\Tr \Bigl(\updelta A_j \upsigma_j ( (hD -A)\cdot \boldupsigma) e^{-tS}\uptheta(-H)  e^{-\lambda S}\Bigr)\\
-&\Tr \Bigl( \updelta A_j \upsigma_j e^{-\lambda S}\uptheta(-H)e^{-\lambda S}( (hD -A)\cdot \boldupsigma)\Bigr).
\end{align*}
Then
$\Tr (\updelta L \theta(-H))=\int \sum_j \Phi_j(x)\updelta A_j $, which implies equality (\ref{2-11}).
\end{proof}

\begin{proposition-foot}\label{prop-2-4}\footnotetext{\label{foot-6} Cf. Proposition~\ref{book_new-prop-27-2-5} of \cite{monsterbook}.}
If for  $\kappa=\kappa^*$
\begin{gather}
\E^* \ge \Weyl_1 - CM
\label{2-18}\\\\
\intertext{with $M\ge C h^{-1}$, then for $\kappa \le \kappa^*(1-\epsilon_0)$}
\frac{1}{\kappa h^2} \int |\partial A|^2\,dx \le C_1M.
\label{2-19}
\end{gather}
\end{proposition-foot}

\begin{proof}
Proof is obvious, also based on the upper estimate $\E^*\le \E(0)\le \Weyl_1+Ch^{-1}$, which is  due to \cite{ivrii:rela}.
\end{proof}

\begin{proposition-foot}\label{prop-2-5}\footnotetext{\label{foot-7} Cf. Proposition~\ref{book_new-prop-27-2-6} of \cite{monsterbook}.}
Let estimate \textup{(\ref{2-19})} be fulfilled and let
\begin{equation}
\varsigma = \kappa M h \le c.
\label{2-20}
\end{equation}
Then for $\tau\le c$

\begin{enumerate}[label=(\roman*), wide, labelindent=0pt]
\item\label{prop-2-5-i}
Operator norm in $\sL^2$ of $(hD)^k \uptheta(\tau -H)$ does not exceed $C$ for $k=0,1$.

\item\label{prop-2-5-ii}
Operator norm in $\sL^2$ of
$(hD)^k\bigl((hD-A)\cdot\boldupsigma\bigr) \uptheta(\tau -H)$ does not exceed $C$ for $k=0$.
\end{enumerate}
\end{proposition-foot}

\begin{proof}
First, let us  repeat of some arguments of the proof of Proposition~\ref{book_new-prop-27-2-6} of \cite{monsterbook}. 
Let $u =\uptheta(\tau-H) f$. Then $\|u\|\le \|f\|$ and since 
\begin{equation}
\|A\|_{\sL^6} \le C\| \partial A\| \le C(\kappa M)^{\frac{1}{2}}h,
\label{2-21}
\end{equation}
we conclude that
\begin{multline*}
\|hD u\| \le \|(hD-A)u\| +\|A u\| \le
\|(hD-A)u\| +C\|A \|_{\sL^6}\cdot \|u\|_{\sL^3}\le\\[2pt]
\|(hD-A)u\| +C(\kappa M)^{\frac{1}{2}}h \|u\|^{1/2}\cdot \|u\|_{\sL^6}^{1/2}\le \\[2pt]
\|(hD-A)u\| +C(\kappa M h)^{\frac{1}{2}} \|u\|^{1/2}\cdot \|h Du\|^{1/2}\le\quad \\[2pt]
\|(hD-A)u\| + \frac{1}{2} \|hDu\| + C\kappa M h  \|u\|;
\end{multline*}
therefore due to (\ref{2-20})
\begin{equation}
\|hD u\| \le 2\|(hD-A)u\| +C\kappa M h  \|u\|.
\label{2-22}
\end{equation}

Further,  $\|T u\|\le c_1\|u\|$ because $H\ge -c$, $|V|\le c$, $|\tau|\le c$; then 
$\|(T+\gamma^{-2})u\le (c_1+\gamma^{-2})\|u\|$ and therefore 
\begin{gather}
(((T+\gamma^{-2})^2-\gamma^{-4})u,u)\le ((c_1+\gamma^{-2})^2-\gamma^{-4})= (2c_1\gamma^{-2} +c_1^2)\|u\|^2,
\notag\\
\intertext{and finally}
(Lu,u)\le C\|u\|^2\label{2-23}\\
\shortintertext{with}
L\coloneqq ((hD-A)\cdot \boldupsigma)^2.
\label{2-24}
\end{gather}
Then, again following the same proof, we conclude that 
\begin{equation}
\|(h D-A)u\|  \le C\|u\|\quad \text{and}\quad \|hDu\|\le C(1+\kappa M h) \|u\|,
\label{2-25}
\end{equation}
provided $\kappa Mh^{1+\delta}\le c$ for sufficiently small $\delta>0$. Therefore under assumption (\ref{2-20})  for $j=0,1$  both Statements~\ref{prop-2-5-i} and \ref{prop-2-5-ii} are proven.
\end{proof}

Thus, in contrast to Proposition~\ref{book_new-prop-27-2-6} of \cite{monsterbook}, we do not have $k=2$ in Statement~\ref{prop-2-5-i}, and $k=1$ in Statement~\ref{prop-2-5-ii} so far and need some extra arguments.

\begin{proposition}\label{prop-2-6}
Assume that $\|V\|_{\sC^2}\le c$. Then 
\begin{equation}
\|[S,V]u\|\le Ch\gamma^2 (\|L^{\frac{1}{2}}u\|+ \|u\|).
\label{2-26}
\end{equation}
\end{proposition}

\begin{proof}
Recall that that $S^2=\gamma^2 L+1$. Therefore $\gamma^2 [L,V]=[S,V]S+S[S,V]$ and then
\begin{gather}
[S,V]=\gamma^2 \int_0^\infty e^{-\lambda S}[L,V]e^{-\lambda S}\,d\lambda.
\label{2-27}\\
\shortintertext{Also}
\|[L,V] w \|\le Ch(\|L^{1/2} w\|+ \|w\|)
\label{2-28}\\
\shortintertext{and}
\|e^{-\lambda S}\|\le e^{-\lambda}.
\label{2-29}
\end{gather}
\end{proof}

\begin{proposition}\label{prop-2-7}
\begin{enumerate}[label=(\roman*), wide, labelindent=0pt]
\item\label{prop-2-7-i}
Assume that $\|V\|_{\sC^2}\le c$ and $|\tau|\le c$. Then the operator norm of  $(hD)^k\uptheta  (\tau-H)$ does not exceed $C$ for $k=0,1,2$.

\item\label{prop-2-7-ii}
Assume that $\|V\|_{\sC^3}\le c$ and $|\tau|\le c$. Then the operator norm of  operators 
$(hD)^k( (hD -A)_x\cdot \boldupsigma)\uptheta  (\tau-H)$ and $(hD)^k\hat{\Phi}_j\uptheta  (\tau-H)$ with
\begin{equation}
\hat{\Phi}_j= \int_0^\infty\upsigma_j ( (hD -A)_x\cdot \boldupsigma) e^{-\lambda S}e(.,.,0)  e^{-\lambda S}\,d\lambda
\label{2-30}
\end{equation}
do not exceed $C$ for $k=0,1,2$.
\end{enumerate}
\end{proposition}

\begin{proof}
\begin{enumerate}[label=(\roman*), wide, labelindent=0pt]
\item\label{pf-2-7-i}
Let $u=\uptheta (\tau-H)f$ with $f\in \sL^2$. Then $u$ satisfies (\ref{2-2}) and $\|Tu\|\le C\|u\|$. Also, which 
implies 
\begin{equation*}
\gamma^{-2}\|Lu\|=\|(T+2\gamma^{-2})Tu\|\le C\gamma^{-2}\|u\| + C\gamma^{-1}\|[T,V]u\|\le C_1\gamma^{-2}\|u\|
\end{equation*}
 due to (\ref{2-26}) .
Then, repeating arguments of the proof of Propositiob~\ref{book_new-27-2-6} of \cite{monsterbook}, we conclude that $\|(hD)^2 u\|\le C\|u\|$, i.e. Statement~\ref{prop-2-7-i}.

\item\label{pf-2-7-ii}
Plugging $(T-V-\tau)u$ instead of $u$ (with $\|(T-V-\tau)u\|\le C\|u\|$) we have
$\|L(T-\tau-V)u\|\le C\|u\|$. Then
\begin{equation*}
\|TLu\|\le C\|Lu\| + C\|[L,V]u\|\le C(\|Lu\|+\|u\|)\le C_1\|u\|.
\end{equation*}
Again plugging $(T-V-\tau)u$ instead of $u$ we have
\begin{multline*}
\gamma^{-2}\|L^2u\| =\| (T+2\gamma^{-2})TLu\| \le \\
\|T(T-V-\tau)Lu\|\le  C \gamma^{-2}\|Lu\| + C\|T[L,V]u\|.
\end{multline*}
Further, the last term does not exceed $Ch\|T V'L^{\frac{1}{2}}u\| + Ch^2\| TV''u\|$ where $V'$ are miscellaneous first derivatives of $V$ and $V''=\Delta V$. Then, the former does not exceed  $C\|Lu\|$, while the latter does not exceed $C\gamma^{-2}\|u\| + h\| \nabla (V''u)\|$, which does not exceed $C\gamma^{-2}\|u\|$.

Therefore $\|L^2 u\|\le C\|u\|$, which implies that $\|L ((hD-A)\cdot \boldupsigma)u\|\le C\|u\|$, which in turn implies that $\|(hD)^2 ((hD-A)\cdot \boldupsigma)u\|\le C\|u\|$ and, finally, $\|(hD)^2\hat{\Phi}_ju\|\le C\|u\|$.
\end{enumerate}
\end{proof}

\begin{corollary}\label{cor-2-8}
\begin{enumerate}[label=(\roman*), wide, labelindent=0pt]
\item\label{cor-2-8-i}
Assume that $\|V\|_{\sC^2}\le c$ and $|\tau|\le c$. Then the operator norm of  $\uptheta  (\tau-H)$ from $\sL^2$ to 
$\sC^{\delta}$ does not exceed $Ch^{-3/2-\delta}$ for $0\le \delta\le \frac{1}{2}$.

\item\label{cor-2-8-ii}
Assume that $\|V\|_{\sC^3}\le c$ and $|\tau|\le c$. Then the operator norm of $\hat{\Phi}_j$ from $\sL^2$ to $\sC^\delta$ do not exceed $Ch^{-3/2-\delta}$ for $0\le \delta\le \frac{1}{2}$. Then $\|\Phi_j\|_{\sC}\le Ch^{-3}$.
\end{enumerate}
\end{corollary}

\begin{corollary}\label{cor-2-9}
\begin{enumerate}[label=(\roman*), wide, labelindent=0pt]
\item\label{cor-2-9-i}
Under assumptions \textup{(\ref{2-18})}--\textup{(\ref{2-20})} $\|A_j\|_{\sC^{2-\delta}}\le C\kappa h^{-1}$ for any $\delta>0$.
\end{enumerate}
\end{corollary}

\section{Microlocal analysis and local theory}
\label{sect-2-3}

\subsection{Microlocal analysis unleashed}
\label{sect-2-3-1}

Then we can apply all arguments of Subsection~\ref{book_new-sect-27-2-2}\footnote{\label{foot-8} Namely of Subsubsections~\emph{\ref{book_new-sect-27-2-2-1}.1. \nameref{book_new-sect-27-2-2-1}\/},  \emph{\ref{book_new-sect-27-2-2-2}.2. \nameref{book_new-sect-27-2-2-2}\/} and \emph{\ref{book_new-sect-27-2-2-3}.3. \nameref{book_new-sect-27-2-2-3}\/}.} of \cite{monsterbook}, even if expression for $\Phi_j$ differs. Indeed, observe first that we can restrict ourselves by $0\le \lambda \le c|\log h|$. Then, using our standard arguments based on the analysis of the propagation of singularities, we can prove that the Tauberian expression with $T=h^{1-\delta}$ for $\Phi_j$ has an error $O(h^{-2})$ provided 
\begin{equation}
V(x)\asymp 1.
\label{2-31}
\end{equation}

Then our standard trick with the freezing coefficients works and with the same $O(h^{-2})$ error we can replace $\Phi_j(x)$ by its Weyl expression, i.e. expression we obtain if replace operators by their symbols, depending on $x$ and $\xi$,  integrating by $d\xi$ and multiplying by $(2\pi h)^{-3}$. However due to skew-symmetry with respect to $\xi- A(x)$, this Weyl expression is $0$, and $\Phi_j (x)=O(h^{-2})$. 

Finally, we can get rid of assumption (\ref{2-31}) by the standard rescaling arguments. We leave all the details to the reader.

\subsection{Local theory and rescaling}
\label{sect-2-3-2}
Then we can apply all arguments of 
Subsection~\ref{book_new-sect-27-2-3}\,\footnote{\label{foot-9} Namely, Subsubsections~\emph{\ref{book_new-sect-27-2-3-1}.1. \nameref{book_new-sect-27-2-3-1}\/} and
\emph{\ref{book_new-sect-27-2-3-2}.2. \nameref{book_new-sect-27-2-3-2}\/}.} of \cite{monsterbook}.
As a result we arrive under assumption (\ref{2-20}) to the trace formula\footnote{\label{foot-10} In the trace formula  ``$\Weyl_1$'' is given by the relativistic expression, $-P^\RTF (W+\nu)$.} with the remainder estimate $O(h^{-1})$ and to estimate $\|\partial A\|=O(\kappa^{1/2}h^{1/2})$. 
 
Moreover, under the standard assumption of the global nature we arrive to the trace formula with the remainder estimate $o(h^{-1})$ (but it will have the Schwinger-type correction term) and to estimate $\|\partial A\|=o(\kappa^{1/2}h^{1/2})$. 

Finally, we an apply all arguments of Subsection~\ref{book_new-sect-27-2-4}\,\footnote{\label{foot-11} Namely, Subsubsection \emph{\ref{book_new-sect-27-2-4-1}.1. \nameref{book_new-sect-27-2-4-1}\/}
and Subsubsection  \emph{\ref{book_new-sect-27-2-4-2}.2. \nameref{book_new-sect-27-2-4-2}\/}.} of \cite{monsterbook}
and we weaken assumption~(\ref{2-20}), recovering the same estimates as before. Again, we leave all the details to the reader.

\chapter{Global trace asymptotics in the case of Coulomb-like singularities}
\label{sect-3}

\section{Estimates to minimizer}
\label{sect-3-1}

Let us return to the original settings, with Coulomb-like singularities and parameters $Z_m, \alpha, \beta$. At the moment we consider the one-particle Hamiltonian. Let us deal first with the vicinity of $\y_m$.

Then we scale like in  Section~\ref{book_new-sect-27-3} of \cite{monsterbook}: 
$x\mapsto Z^{\frac{1}{3}}x$,  $\tau \mapsto Z^{-\frac{4}{3}}\tau$, $A\mapsto Z^{-\frac{2}{3}}A$, 
$\beta\mapsto \beta Z^{\frac{3}{3}}$, arriving to the semiclassical problem with Coulomb singularities  
$z_m|x-\y_m|^{-1}$ ($z_m=Z_mZ^{-1}$),  with $h= Z^{-\frac{1}{3}}$ and with  $\kappa =\alpha Z^{\frac{2}{3}}$. In particular,
$\E^*$ is a minimum with respect to $A$ of 
\begin{equation}
\E(A)\coloneqq \Tr (H_{A,W+\tau}^-)+\kappa^{-1}h^{-2}\|\partial A\|^2.
\label{3-1}
\end{equation}

Let us follow arguments of Subsubsection~\emph{\ref{book_new-sect-27-3-2-1}.1\ \nameref{book_new-sect-27-3-2-1}\/}.
Observe first that  the estimate from above is 
\begin{equation}
\E^* \le h^{-3}\int \Weyl_1(x)\,dx + Ch^{-2};
\label{3-2}
\end{equation}
we simply take $A=0$ and refer to \cite{ivrii:rela}\footnote{\label{foot-12} In the original settings the remainder estimate  would be $O(Z^2)$ exactly as in the non-relativistic case.}.

Consider now estimate from below and apply $\ell$-admissible partition exactly like in Subsection~\ref{book_new-sect-27-3-2} of \cite{monsterbook}. Then, according to the previous section, for any element of partition with 
$\ell \ge ch^{-2}$ ($\ell \ge cZ_m^{-1}$ in the original settings) its contribution is estimated from below by the corresponding Weyl expression minus $Ch^{-1}\zeta ^2  \times \zeta \ell =Ch^{-1}\zeta^3\ell^{-1}$ and summation with respect to these elements returns $\cE^\TF$ minus $Ch^{-1}\zeta^3\ell |_{\ell = h^{-2}}$, i.e. 
\begin{equation}
h^{-3}\int \Weyl_1(x)\,dx - Ch^{-2}
\label{3-3}
\end{equation}
because the contribution of the zone $\cZ_0= \{x\colon \ell (x)\le ch^{-2}\}$ to the main term is $O(h^{-2})$.

On the other hand, the contribution of $\cZ_0$ is $-Ch^{-2}$. Indeed, scale first $x\mapsto h^{-2}x$, $\tau\mapsto h^{2}\tau$, $h\mapsto 1$, $A\mapsto h A$, $\beta \mapsto \gamma=\beta h^{-1}$, and the Coulomb singularity remains the same while the magnetic energy becomes $\kappa^{-1}\|\partial A\|^2$. Observe that $\gamma=\beta_\orig Z_m$\,\footnote{\label{foot-13} Considering vicinity of $\y_m$ it is more convenient to take the original rescaling with $Z$ replaced by $Z_m$, and therefore $z_m=1$.}, so (\ref{1-8}), (\ref{1-11}) become
\begin{equation}
\beta \le 2\pi^{-1}-\epsilon,\qquad \kappa \le \kappa^*(2\pi^{-1}-\beta).
\label{3-4}
\end{equation}
Then we can apply a ``singular magnetic Daubechies inequality'' (\ref{A-3}) and repeat all arguments of the regular case in a simple case of $h=1$. There will be an extra terms $O(1)$ and $-C(1-\pi \gamma /2)^{-\frac{3}{2}}\|\partial A\|^2$ and that latter term requires (\ref{1-11}).

Now we conclude that Proposition~\ref{book_new-prop-27-3-1} of \cite{monsterbook} holds:

\begin{proposition-foot}\footnotetext{\label{foot-14} Cf. Proposition~\ref{book_new-prop-27-3-1} of \cite{monsterbook}.}\label{prop-3-1}
In our framework
$\kappa \le \kappa^*$. Then the near-minimizer $A$ satisfies
\begin{gather}
|\int \bigl(\tr e_{A,1}(x,x,0)-\Weyl_1(x)\bigr)\,dx|\le Ch^{-2}
\label{3-5}\\\\
\shortintertext{and}
\|\partial A\| \le C\kappa^{\frac{1}{2}}.
\label{3-6}
\end{gather}
\end{proposition-foot}

It allows us to repeat arguments of the proof Proposition~\ref{prop-2-2} and to prove

\begin{proposition-foot}\footnotetext{\label{foot-15} Cf. Proposition~\ref{book_new-prop-27-3-2} of \cite{monsterbook}.}\label{prop-3-2}
In our framework there exists a minimizer  $A$\,\footref{foot-4}.
\end{proposition-foot}

Now we can repeat arguments of  Subsubsection~\emph{\ref{book_new-sect-27-3-2-2}.2\ \nameref{book_new-sect-27-3-2-2}\/} of \cite{monsterbook}, albeit with the right-hand expression of (\ref{book_new-27-3-14}) given now by (\ref{2-11}) and to prove the claim (\ref{book_new-27-3-28}), which is marginally stronger than
\begin{equation}
\|\partial^2 A\|_{\sL^\infty (B(0,1-\epsilon))}\le C\kappa^{\frac{1}{2}}h^{-\delta}.
\label{3-7}
\end{equation}

Then we can repeat arguments of  Subsubsection~\emph{\ref{book_new-sect-27-3-2-3}.3\ \nameref{book_new-sect-27-3-2-3}\/}  of \cite{monsterbook} and recover Propositions~\ref{book_new-prop-27-3-4}, \ref{book_new-prop-27-3-6} and \ref{book_new-prop-27-3-7}, estimating $A$ and its derivatives as $\ell (x)\lesssim 1$:

\begin{proposition-foot}\footnotetext{\label{foot-16} Cf. Proposition~\ref{book_new-prop-27-3-7}\ref{book_new-prop-27-3-7-i} of \cite{monsterbook}.} \label{prop-3-3}
In our framework if $\ell (x)\ge \ell_*\coloneqq h^2$, then
\begin{gather}
|A|\le C\kappa \ell^{-\frac{1}{2}},\qquad
|\partial A|\le C\kappa \ell ^{-\frac{3}{2}}
\label{3-8}\\\\
\shortintertext{and}
|\partial A (x)-\partial A(y)|\le C_\theta\kappa \ell^{-\frac{3}{2}-\theta}|x-y|^\theta \qquad \text{as\ \ } |x-y|\le \frac{1}{2}\ell (x)
\label{3-9}
\end{gather}
for any $\theta \in (0,1)$.
\end{proposition-foot}

Consider now the non-semiclassical  zone $\{x\colon \ell(x)\lesssim \ell_*\}$, which contains the relativistic zone
$\{x\colon \ell(x)\lesssim \bar{\ell}\coloneqq \gamma h\}$. Using arguments of the proof of Proposition~\ref{rel-prop-3-4} of \cite{ivrii:rela}, but additionally taking care of the magnetic field using arguments of the proofs of 
Propositions~\ref{prop-2-5},~\ref{prop-2-6} and~\ref{prop-2-7} (we leave all details to the reader),
we arrive to 

\begin{proposition}\label{prop-3-4}
In our framework $H_{W,A}\ge C_0\ell_*^{-1}$ and $e(x,x,\lambda) \le C\ell_*^{-3}$ for $\ell(x)\le c\ell_*$ and 
$|\lambda|\le Ch^{-2}$.
\end{proposition}

\begin{remark}\label{rem-3-5}
\begin{enumerate}[label=(\roman*), wide, labelindent=0pt]
\item\label{rem-3-5-i}
Then in the original settings $H_{W,A}\ge C_0Z^{-2}$ and $e(x,x,\lambda) \le CZ^3$ for $\ell(x)\le cZ^{-1}$ and 
$|\lambda|\le C_0Z^2$.
\item\label{rem-3-5-ii}
We have a better estimate than (\ref{rel-3-11}) of \cite{ivrii:rela} due to assumptions (\ref{1-8}) and (\ref{1-11}).
\end{enumerate}
\end{remark} 

Next, we  follow arguments of Subsubsection~\emph{\ref{book_new-sect-27-3-2-4}.4\ \nameref{book_new-sect-27-3-2-4}\/} of \cite{monsterbook}  and prove (again, leaving details to the reader)

\begin{proposition-foot}\footnotetext{\label{foot-17} Cf. Proposition~\ref{book_new-prop-27-3-9} of \cite{monsterbook}.}
\label{prop-3-6}
In our framework 
\begin{gather}
|A|\le C\kappa \ell^{-2},\qquad
|\partial A|\le C\kappa \ell ^{-3}\label{3-10}\\\\
\shortintertext{and}
|\partial A (x)-\partial A(y)|\le C_\theta\kappa \ell^{-3-\theta}|x-y|^\theta \qquad \text{as\ \ } |x-y|\le \frac{1}{2}\ell (x)
\label{3-11}
\end{gather}
if $\ell(x)\ge 1$ (for all $\theta \in (0,1)$).
\end{proposition-foot}
 
\section{Trace estimates}
\label{sect-3-2}

Next we can go after trace asymptotics. Recall that we are dealing with the rescaled operator. Let $a$ be the minimal distance between nuclei (after rescaling), capped by $1$; recall that $a\ge \ell_*$. 

After we estimated $A$ for $\ell(x)\lesssim 1$ in Proposition~\ref{prop-3-1} and $e(x,x,\lambda)$ for 
$\ell(x)\lesssim \ell_*=h^{-2}$, we can apply arguments of Subsection~\ref{book_new-sect-27-3-3} of \cite{monsterbook} and arrive to 

\begin{proposition-foot}\footnotetext{\label{foot-18} Cf. Proposition~\ref{book_new-prop-27-3-11} of \cite{monsterbook}.}\label{prop-3-7}
In our framework let $\psi $ be $a$-admissible and supported in $\frac{1}{2}a$-vicinity of $\y_m$, let $\varphi $ be $\ell_*$-admissible, supported in $2\ell_*$-vicinity and equal $1$ in $\ell_*$-vicinity of $\y_m$, and let 
$V^0=Z_m|x|^{-1}$. Then
\begin{multline}
\Tr \bigl(\psi (H^-_{A,V} - H^-_{A,V^0})\psi\bigr)= \\[3pt]
\int \bigl(\Weyl_1(x) -\Weyl_1^0(x)\bigr)(1-\varphi(x))\, \,dx + O\bigl(a^{-\frac{1}{3}}h^{-\frac{4}{3}}\bigr).
\label{3-12}
\end{multline}
\end{proposition-foot}

\begin{remark}\label{rem-3-8}
Here and in Proposition~\ref{prop-3-9} $\Weyl$ and $\Weyl_1$ are  defined for the relativistic operator (i.e.
$\Weyl=P^{\RTF\,\prime} (V)$ and $\Weyl=-P^\RTF (V)$), but following arguments of \ref{rel-prop-3-6} of \cite{ivrii:rela}, we can replace it by those for non-relativistic operator (i.e.
$\Weyl=P^{\TF\,\prime} (V)$ and $\Weyl=-P^\TF (V)$) and then skip the factor 
$(1-\varphi(x))$.
\end{remark}

Moreover, applying arguments of Subsection~\ref{book_new-sect-27-3-4} of \cite{monsterbook} we arrive to

\begin{proposition-foot}\footnotetext{\label{foot-19} Cf. Proposition~\ref{book_new-prop-27-3-16} of \cite{monsterbook}.}\label{prop-3-9}
\begin{enumerate}[label=(\roman*), wide, labelindent=0pt]
\item \label{prop-3-9-i}
In the framework of Proposition~\ref{prop-3-7}
\begin{multline}
\Tr \bigl(\psi (H^-_{A,V} - H^-_{A,V^0})\psi\bigr)= \\[3pt]
\int \bigl(\Weyl_1(x) -\Weyl_1^0(x)\bigr)\,\psi^2(x)(1-\varphi(x)) \,dx + O\bigl(h^{-\frac{4}{3}}a^{-\frac{1}{3}}\kappa |\log \kappa|^{\frac{1}{3}}+
h^{-1}a^{-\frac{1}{2}}\bigr).
\label{3-13}
\end{multline}

\item \label{prop-3-9-ii}
In particular, if
\begin{equation}
\kappa \le c a^{-\frac{1}{6}}h^{\frac{1}{3}} |\log ah^{-2}|^{-\frac{1}{3}},
\label{3-14}
\end{equation}
then the error in \textup{(\ref{3-13})} does not exceed $Ch^{-1}a^{-\frac{1}{2}}$ exactly as in the case without magnetic field.
\end{enumerate}
\end{proposition-foot}

Next consider the case of exactly Coulomb potential $V=Z|x|^{-1}$ and $\nu=0$. Then

\begin{proposition-foot}\footnotetext{\label{foot-20} Cf. Proposition~\ref{book_new-prop-27-3-18} of \cite{monsterbook}.}\label{prop-3-10}
Let $V=Z |x|^{-1}$, $h>0 $, $Z>0$, and \textup{(\ref{1-8})} and \textup{(\ref{1-11})} be fulfilled.  Then
\begin{enumerate}[label=(\roman*), wide, labelindent=0pt]
\item\label{prop-3-10-i}
The following limit exists\,\footnote{\label{foot-21} Cf. \textup{(\ref{book_new-27-3-71})} of \cite{monsterbook} and \textup{(\ref{rel-3-18})} of \cite{ivrii:rela}.}
\begin{multline}
\lim_{r\to\infty} \biggl( \inf_A
\Bigl( \Tr \bigl( (\phi_r H_{A,V} \phi_r)^-\bigr) + \frac{1}{\kappa h^2} \int |\partial A|^2 \,dx  \Bigr)\\
- \int  \Weyl_1 (x)\phi_r ^2(x)\,dx \biggr) \Fed 2 Z^2 h^{-2} S(Z\kappa, Z\beta).
\label{3-15}
\end{multline}

\item\label{prop-27-3-18-ii}
And it coincides with \textup{(\ref{book_new-27-3-72})}
and  also with \textup{(\ref{book_new-27-3-73})} of \cite{monsterbook}.

\item\label{prop-3-10-iv}
We also can replace in Statement~\ref{prop-3-10-i}
$\Tr \bigl( (\phi_r H_{A,V} \phi_r)^-)$ by
$\Tr \bigl( \phi_r H^-_{A,V} \phi_r\bigr)$.
\end{enumerate}
Here $\phi \in \sC_0^\infty (B(0,1))$, $\phi=1$ in $B(0,\frac{1}{2})$, $\phi_r=\phi(x/r)$ and $\Weyl$ and $\Weyl_1$ are defined for non-relativistic operator.
\end{proposition-foot}

Then we also arrive to 

\begin{proposition-foot}\footnotetext{\label{foot-22} Cf. Proposition~\ref{book_new-prop-27-3-20}  of \cite{monsterbook} and Remark~\ref{rel-rem-3-8} of \cite{ivrii:rela}.}\label{prop-3-11}
In the framework of Proposition~\ref{prop-3-10} for $0<\kappa <\kappa'$, $\beta<\beta'$
\begin{gather}
S(\kappa',\beta)\le S(\kappa,\beta) \le S(\kappa',\beta) + C \kappa '(\kappa^{-1}-\kappa'^{-1}),\label{3-16}\\
S(\kappa,\beta)\le S(\kappa',\beta').
\label{3-17}
\end{gather}
\end{proposition-foot}

Then, in the ``atomic'' case $M=1$ we arrive instantly to the following theorem:

\begin{theorem-foot}\footnotetext{\label{foot-23} Cf. Theorem~\ref{book_new-thm-27-3-22} of \cite{monsterbook}
and Propositions~\ref{rel-prop-3-9} and~\ref{rel-prop-3-10} of \cite{ivrii:rela}.}\label{thm-3-12}
Let $M=1$ and \textup{(\ref{1-8})} and \textup{(\ref{1-11})} be fulfilled. Then
\begin{enumerate}[label=(\roman*), wide, labelindent=0pt]
\item\label{thm-3-12-i}
The following asymptotics holds
\begin{equation}
\E^* = \int \Weyl_1(x)\,dx + 2 z^2 S(z \kappa, z\beta) h^{-2}+
O(h^{-\frac{4}{3}}\kappa |\log \kappa|^{\frac{1}{3}}+h^{-1}).
\label{3-18}
\end{equation}

\item\label{thm-3-12-ii}
If $\kappa =o (h^{\frac{1}{3}}|\log h|^{-\frac{1}{3}})$,  then
\begin{equation}
\E^* = \int \Weyl^*_1(x)\,dx + 2 z^2 S(z \kappa, z\beta) h^{-2}+ o( h^{-1}),
\label{3-19}
\end{equation}
in which case $\Weyl^*_1$ must in addition to $-h^{-3}P^\TF (W+\nu)$ contain  the Schwinger correction, and also the relativistic correction.
\end{enumerate}
\end{theorem-foot}

Next, using arguments Subsection~\ref{book_new-sect-27-3-6} of \cite{monsterbook}, in particular, decoupling of singularities (which is needed only in the case ofthe self-generated magnetic field), we arrive to

\begin{theorem-foot}\footnotetext{\label{foot-24} Cf. Theorem~\ref{book_new-thm-27-3-24} of \cite{monsterbook}.}\label{thm-3-13}
Let $M\ge 2$, $\kappa \le \kappa^*$ and  \textup{(\ref{1-8})} and \textup{(\ref{1-11})} be fulfilled. Then
\begin{enumerate}[label=(\roman*), wide, labelindent=0pt]
\item\label{thm-3-13-i}
The following asymptotics holds
\begin{equation}
\E^* = \int \Weyl_1(x)\,dx + 2 \sum_{1\le m\le M} z_m^2 S(z_m \kappa,z_m\beta) h^{-2}+ O(R_1+R_2)
\label{3-20}
\end{equation}
with
\begin{align}
R_1\, =\,&\left\{\begin{aligned}
&h^{-1}+\kappa |\log \kappa|^{\frac{1}{3}}h^{-\frac{4}{3}}  &&\text{if\ \ } a\ge 1, \\
&a^{-\frac{1}{2}}h^{-1}+\kappa |\log \kappa|^{\frac{1}{3}}a^{-\frac{1}{3}}h^{-\frac{4}{3}}\qquad
&&\text{if\ \ } h^2\le a\le 1
\end{aligned}\right.
\label{3-21}\\\\
\shortintertext{and}
R_2=\kappa h^{-2}
&\left\{\begin{aligned}
&a^{-3} &&\text{if\ \ } a\ge |\log h|^{\frac{1}{3}},\\
&|\log h^2a^{-1}|^{-1} \qquad &&\text{if\ \ }  h^2\le a \le |\log h|^{\frac{1}{3}}.
\end{aligned}\right.
\label{3-22}
\end{align}

\item\label{thm-3-13-ii}
If $\kappa =o( h^{\frac{1}{3}}|\log h|^{-\frac{1}{3}})$, $\kappa a^{-3}=o(h)$ and $a^{-1}=o(1)$, then
\begin{equation}
\E^* = \int \Weyl^*_1(x)\,dx + 2 \sum_{1\le m\le M} z_m^2 S(z_m \kappa, z_m\beta) h^{-2}+ o(h^{-1}).
\label{3-23}
\end{equation}
\end{enumerate}
\end{theorem-foot}

\chapter{Main results}
\label{sect-4}

\section{Asymptotics of the ground state energ}
\label{sect-4-1}

Now we can apply arguments of Section~\ref{book_new-sect-27-4}. In addition to (\ref{1-8}) and (\ref{1-11}) we assume that
\begin{align}
&d\coloneqq \min_{1\le m<m'\le M}|\y_m-\y_{m'}|\ge Z^{-1},\label{4-1}\\
&N\asymp Z_1\asymp \ldots\asymp Z_M.\label{4-2}
\end{align}
Then the estimates from below follow immediately from the trace asymptotics, while for the estimate from above we need also estimate $\N$ and miscellaneous $\D$-terems. We leave all the details to the reader.
 
\begin{theorem-foot}\footnotetext{\label{foot-25} Cf. Theorem~\ref{book_new-thm-27-4-3}  of \cite{monsterbook}.}\label{thm-4-1}
\begin{enumerate}[label=(\roman*), wide, labelindent=0pt]
\item\label{thm-4-1-i}
Under assumptions \textup{(\ref{1-8})}, \textup{(\ref{1-11})}, \textup{(\ref{4-1})} and \textup{(\ref{4-2})} the following asymptotics holds
\begin{equation}
\E^*_N = \cE^\TF_N + \sum_{1\le m\le M} 2Z_m^2 S(\alpha Z_m, \beta Z_m) + O\bigl(Z^{\frac{4}{3}}(R_1+R_2)\bigr)
\label{4-3}
\end{equation}
with $R_1$ and $R_2$ defined by \textup{(\ref{3-21})} and \textup{(\ref{3-22})} respectively with $\kappa =\alpha Z$, $h=Z^{-\frac{1}{3}}$ and $a=Z^{\frac{1}{3}}d$, $d$ is defined by \textup{(\ref{4-1})}, $d=\infty$ for $M=1$.

\item\label{thm-4-1-ii}
In particular, under assumption $d\gtrsim Z^{-\frac{1}{3}}$
the following asymptotics holds
\begin{multline}
\E^*_N = \cE^\TF_N + \sum_{1\le m\le M} 2Z_m^2 S(\alpha Z_m,\beta Z_m) +\\
O\bigl(\alpha |\log (\alpha Z)|^{\frac{1}{3}}  Z^{\frac{25}{9}}+ Z^{\frac{5}{3}}
+ \alpha d^{-3}Z ^2 \bigr).
\label{4-4}
\end{multline}
\end{enumerate}
\end{theorem-foot}

\begin{theorem-foot}\footnotetext{\label{foot-26} Cf. Theorem~\ref{book_new-thm-27-4-4}  of \cite{monsterbook}.}\label{thm-4-2}

\begin{enumerate}[label=(\roman*), wide, labelindent=0pt]
\item\label{thm-4-2-i}
Let assumptions  \textup{(\ref{1-8})}, \textup{(\ref{1-11})}, \textup{(\ref{4-1})} and \textup{(\ref{4-2})}  be fulfilled and let $\Psi=\Psi_{\mathbf{A}}$ be a ground state for a near optimizer $\mathbf{A}$ of the original multiparticle problem. Then
\begin{equation}
\D(\rho_\Psi-\rho^\TF,\,  \rho_\Psi-\rho^\TF)\le CZ^{\frac{5}{3}}.
\label{4-5}
\end{equation}
\item\label{thm-4-2-ii}
Furthermore, if $d\ge Z^{-\frac{1}{3}}$, then
\begin{equation}
\D(\rho_\Psi-\rho^\TF,\,  \rho_\Psi-\rho^\TF)\le CZ^{\frac{5}{3}}\bigl(Z^{-\delta} +(dZ^{\frac{1}{3}})^{-\delta}+ (\alpha Z)^{\delta}\bigr).
\label{4-6}
\end{equation}
\end{enumerate}
\end{theorem-foot}

\begin{theorem-foot}\footnotetext{\label{foot-27} Cf. Theorem~\ref{book_new-thm-27-4-6}  of \cite{monsterbook}.}\label{thm-4-3}
Let assumptions  \textup{(\ref{1-8})}, \textup{(\ref{1-11})}, \textup{(\ref{4-1})} and \textup{(\ref{4-2})}  be fulfilled, and let
$\alpha  \le Z^{-\frac{10}{9}}|\log Z|^{-\frac{1}{3}}$. Then
\begin{multline}
\E^*_N = \cE^\TF_N + \sum_{1\le m\le M} 2Z_m^2 S(\alpha Z_m,\beta Z_m) +\Dirac+\Schwinger +\mathsf{RCT}+\\
O\bigl(\alpha |\log (\alpha Z)|^{\frac{1}{3}}  Z^{\frac{25}{9}}+ Z^{\frac{5}{3}-\delta} +  \alpha d^{-3}Z ^2 \bigr)
\label{4-7}
\end{multline}
where $\Dirac$ and $\Schwinger$ are \emph{Dirac\/} and \emph{Schwinger correction terms\/} defined exactly as in non-magnetic non-relativistic case by \textup{(\ref{book_new-25-1-29})} and  \textup{(\ref{book_new-25-1-30})} of \cite{monsterbook} respectively, and $\mathsf{RCT}$ is relativistic correction term, defined as in the non-magnetic case by \textup{(\ref{rel-3-23})} of \cite{ivrii:rela}.
\end{theorem-foot}

\begin{theorem-foot}\footnotetext{\label{foot-28} Cf. Theorem~\ref{book_new-thm-27-4-7} of \cite{monsterbook}.}\label{thm-4-4}
Let assumptions  \textup{(\ref{1-8})}, \textup{(\ref{1-11})} and \textup{(\ref{4-2})} be fulfilled. Let us consider $\y_m=\y_m^*$ minimizing the full energy
\begin{gather}
\widehat{\E}^*_N\coloneqq   \E^*_N +\sum _{1\le m <m'\le M} Z_mZ_{m'}|\y_m-\y_{m'}|^{-1}.
\label{4-8}\\\\
\shortintertext{Then}
d\ge \min \bigl( Z^{-\frac{5}{21}+\delta},\,
Z^{-\frac{5}{21}}(\alpha Z)^{-\delta},\,
\alpha^{-\frac{1}{4}}Z^{-\frac{1}{2}}\bigr)
\label{4-9}
\end{gather}
and in the remainder estimates in \textup{(\ref{4-4})} and~\textup{(\ref{4-7})} one can skip $d$-connected terms; so we arrive to
\begin{equation}
\E^*_N = \cE^\TF_N + \sum_{1\le m\le M} 2Z_m^2 S(\alpha Z_m,\beta Z_m) +
O\bigl(\alpha |\log (\alpha Z)|^{\frac{1}{3}}  Z^{\frac{25}{9}}+ Z^{\frac{5}{3}}\bigr)
\label{4-10}
\end{equation}
and
\begin{multline}
\E^*_N = \cE^\TF_N + \sum_{1\le m\le M} 2Z_m^2 S(\alpha Z_m) +\Dirac+\Schwinger +\mathsf{RCT}+\\
O\bigl(\alpha |\log (\alpha Z)|^{\frac{1}{3}}  Z^{\frac{25}{9}}+ Z^{\frac{5}{3}-\delta} \bigr)
\label{4-11}
\end{multline}
respectively and also the same asymptotics with $\widehat{\E}^*_N$ and $\widehat{\cE}^\TF_N$ instead of $\E^*_N$ and $\cE^\TF_N$.
\end{theorem-foot}

\section{Related problems}
\label{sect-4-2}

After Theorem~\ref{thm-4-6} is proven, we can apply arguments of Sections~\ref{book_new-sect-25-5} and \ref{book_new-sect-25-6} of \cite{monsterbook}.

\begin{theorem-foot}\footnotetext{\label{foot-29} Cf. Theorem~\ref{book_new-thm-27-5-1}  of \cite{monsterbook}.}\label{thm-4-5}
Let assumptions  \textup{(\ref{1-8})}, \textup{(\ref{1-11})} and \textup{(\ref{4-2})} be fulfilled.
\begin{enumerate}[label=(\roman*), wide, labelindent=0pt]
\item\label{thm-4-5-i}
In the framework of the fixed nuclei model let us assume that \newline $\I^*_N\coloneqq   \E^*_{N-1}-\E^*_N>0$. Then
\begin{equation}
(N-Z)_+\le CZ^{\frac{5}{7}}
\left\{\begin{aligned}
&1 \qquad &&\text{if\ \ } d\le Z^{-\frac{1}{3}},\\
&Z^{-\delta}+ (dZ^{\frac{1}{3}})^{-\delta}+(\alpha Z)^\delta
&&\text{if\ \ } d\ge Z^{-\frac{1}{3}}.
\end{aligned}\right.
\label{4-12}
\end{equation}

\item\label{thm-4-5-ii}
In particular, for a single atom and for molecule with
$d\ge Z^{-\frac{1}{3}+\delta}$
\begin{equation}
(N-Z)_+\le Z^{\frac{5}{7}}\bigl(Z^{-\delta}+(\alpha Z)^\delta\bigr).
\label{4-13}
\end{equation}

\item\label{thm-4-5-iii}
In the framework of  the free nuclei model let us assume that
$\widehat{\I}^*_N\coloneqq   \widehat{\E}^*_{N-1}-\widehat{\E}^*_N>0$. Then estimate \textup{(\ref{4-13})} holds.
\end{enumerate}
\end{theorem-foot}

\begin{theorem-foot}\footnotetext{\label{foot-30} Cf. Theorem~\ref{book_new-thm-27-5-2}  of \cite{monsterbook}.}\label{thm-4-6}
Let assumptions  \textup{(\ref{1-8})}, \textup{(\ref{1-11})} and \textup{(\ref{4-2})} be fulfilled and let
$N\ge Z-C_0 Z^{\frac{5}{7}}$. Then
\begin{enumerate}[label=(\roman*), wide, labelindent=0pt]
\item\label{thm-4-6-i}
In the framework of the fixed nuclei model
\begin{equation}
\I^*_N \le  CZ^{\frac{20}{21}}.
\label{4-14}
\end{equation}
\item\label{thm-4-6-ii}
In the framework of the free nuclei model with
$N\ge Z-C_0 Z^{\frac{5}{7}}\bigl(Z^{-\delta}+{\alpha Z}^{\delta}\bigr)$
\begin{equation}
\widehat{\I}_N^*\coloneqq   \widehat{\E}^*_{N-1}- \widehat{\E}^*_{N-1} \le Z^{\frac{20}{21}}\bigl(Z^{-\delta'}+(\alpha Z)^{\delta'}\bigr).
\label{4-15}
\end{equation}
\end{enumerate}
\end{theorem-foot}

\begin{theorem-foot}\footnotetext{\label{foot-31} Cf. Theorem~\ref{book_new-thm-27-5-3} of \cite{monsterbook}.}\label{thm-4-7} Let assumptions  \textup{(\ref{1-8})}, \textup{(\ref{1-11})} and \textup{(\ref{4-2})} be fulfilled and let  $N\le Z-C_0 Z^{\frac{5}{7}}$. Then in the framework of the fixed nuclei model under assumption $b\ge C_1(N-Z)^{-\frac{1}{3}}$
\begin{equation}
(\I^*_N +\nu )_+\le C (Z-N)^{\frac{17}{18}}Z^{\frac{5}{18}}
\left\{\begin{aligned}
&1 \qquad &&\text{if\ \ } d\le Z^{-\frac{1}{3}},\\
&Z^{-\delta}+ (dZ^{\frac{1}{3}})^{-\delta}
&&\text{if\ \ } d\ge Z^{-\frac{1}{3}}.
\end{aligned}\right.
\label{4-16}
\end{equation}
\end{theorem-foot}

\begin{theorem-foot}\footnotetext{\label{foot-32} Cf. Theorem~\ref{book_new-thm-27-5-6} of \cite{monsterbook}.}\label{thm-4-8}
Let assumptions  \textup{(\ref{1-8})}, \textup{(\ref{1-11})}. Then in the framework of free nuclei model with $M\ge 2$ the stable molecule does not exist unless
\begin{equation}
Z-N \le Z^{\frac{5}{7}}\bigl( Z^{-\delta}+(\alpha Z)^\delta\bigr).
\label{4-17}
\end{equation}
\end{theorem-foot}

\begin{appendices}
\chapter{Appendix}

In this section we reproduce from~\cite{EFS2}: two new Lieb-Thirring type inequalities
for the relativistic kinetic energy with a magnetic field.

\begin{theorem-foot}\footnotetext{\label{foot-33} Theorem 2.2 of \cite{EFS2}.}\label{thm-A-1}
There exists a universal constant $C>0$ such that for any positive number $\gamma>0$, for any potential $V$ with $V_{+} \in \sL^{5/2}\cap \sL^{4}(\bR^3)$, and magnetic field $B = \nabla \times A \in \sL^2({\bR}^3)$, we have
\begin{multline}
\Tr \Bigl(\bigl( \sqrt{\gamma^{-2} (D-A)\cdot\boldupsigma)^2 + \gamma^{-4}} - \gamma^{-2} - U(x) \bigl)^-\Bigr)  \ge  \\[3pt]
 - C\biggl\{ \int U_{+}^{5/2}\,dx + \gamma^3  \int U_+^4 \,dx
+\Bigl( \int |\nabla \times A|^2\,dx \Bigr)^{3/4}
\Bigl(\int U_+^4\,dx\Bigr)^{1/4}\biggr\}.
\label{A-1}
\end{multline}
\end{theorem-foot}

Notice that
Theorem~\ref{thm-A-1} reduces to the well-known Daubechies inequality in the case $A=0$ \cite{Dau}.

For the Schr\"odinger case,
the Daubechies inequality was generalized (and improved to incorporate a critical Coulomb singularity)
 to non-zero $A$   in \cite{FLS} by using diamagnetic techniques. Theorem~\ref{thm-A-1}
is the generalization of the Daubechies inequality for the Pauli operator, in which case there is
no diamagnetic inequality.
Moreover, in the $\gamma\to 0$ limit, (\ref{A-1}) converges to the magnetic Lieb-Thirring inequality
for the Pauli operator \cite{LLS}
since
\begin{equation}
\sqrt{\gamma^{-2} (D-A)\cdot\boldupsigma)^2 + \gamma^{-4}} - \gamma^{-2}\to \frac{1}{2} (D-A)\cdot\boldupsigma)^2, \qquad \gamma\to 0.
\label{A-2}
\end{equation}

Theorem~\ref{thm-A-1} does not cover the case of a Coulomb singularity. The next result shows that  for $\gamma$ smaller than the critical value $2/\pi$, the Coulomb singularity can be included.  

\begin{theorem-foot}\footnotetext{\label{foot-34} Theorem 2.3 of \cite{EFS2}.}\label{thm-A-2}
Let $\phi_r$ be a real function satisfying $\supp \phi_r \subset \{|x|\leq r\}$,
$\| \phi_r \|_{\infty} \leq 1$.
There exists a constant $C>0$ such that if $\gamma \in (0,2/\pi)$, then
\begin{multline}\
\Tr \Bigl(\phi_r \bigl( \sqrt{\gamma^{-2} (D-A)\cdot\boldupsigma)^2 + \gamma^{-4}} - \gamma^{-2} -
 \frac{1}{|x|}- U \bigr) \phi_r \Bigr)^-  \\
\ge
- C \biggl\{ \eta^{-3/2} \int |\nabla \times A|^2\,dx + \eta^{-3} r^3 + \eta^{-3/2} \int U_{+}^{5/2}\,dx +
\eta^{-3} \gamma^3 \int U_{+}^4\,dx \\
+ \Bigl(\int |\nabla \times A|^2\,dx \Bigr)^{3/4}
\Bigl(\int U_{+}^4\,dx \Bigr)^{1/4} \biggr\},
\label{A-3}
\end{multline}
where $\eta := \frac{1}{10}(1-(\pi\gamma/2)^2)$.
\end{theorem-foot}
\end{appendices}

\end{document}

%% file: defines.tex
\DeclareTextCommand{\textvartheta}{PU}{\83\321}

\DeclareMathAlphabet{\mathpzc}{OT1}{pzc}{m}{it}

\newcommand{\cE}{\mathcal{E}}

\newcommand{\cZ}{\mathcal{Z}}

\newcommand{\sC}{\mathscr{C}}

\newcommand{\sH}{\mathscr{H}}

\newcommand{\sL}{\mathscr{L}}

\newcommand{\E}{{\mathsf{E}}}
\newcommand{\TF}{{\mathsf{TF}}}

\newcommand{\Weyl}{{\mathsf{Weyl}}}

\newcommand{\Dirac}{{\mathsf{Dirac}}}
\newcommand{\Schwinger}{{\mathsf{Schwinger}}}

\newcommand{\orig}{{\mathsf{orig}}}

\newcommand{\D}{{\mathsf{D}}}

\newcommand{\sfH}{{\mathsf{H}}}

\newcommand{\N}{{\mathsf{N}}}
\newcommand{\I}{{\mathsf{I}}}

\newcommand{\x}{{\mathsf{x}}}
\newcommand{\y}{{\mathsf{y}}}

\newcommand{\loc}{{{\mathsf{loc}}}}

\newcommand{\bC}{{\mathbb{C}}}

\newcommand{\bR}{{\mathbb{R}}}

\newcommand{\fH}{{\mathfrak{H}}}

\newcommand{\boldupsigma}{{\boldsymbol{\upsigma}}}

\newcommand{\3}{{|\!|\!|}}

\newcommand{\Fed}{=\mathrel{\mathop:}}

\renewcommand{\Re}{\operatorname{Re}}       

\newcommand{\Spec}{\operatorname{Spec}}

\newcommand{\supp}{\operatorname{supp}}

\newcommand{\tr}{\operatorname{tr}}
\newcommand{\Tr}{\operatorname{Tr}}

\newenvironment{claim*}[1][{}]{\vglue10pt
\begin{trivlist}
\item[]}{\vglue10pt\end{trivlist}}

\newenvironment{phantomequation}[1][]{\refstepcounter{equation}}{}
\newcounter{note}

%% file: Relativistic.II.bbl
\begin{thebibliography}{Ivr}

\bibitem[Bach]{Bach}
V. Bach. \emph{Error bound for the Hartree-Fock energy of atoms and molecules\/}.
Commun. Math. Phys. 147:527--548 (1992).

\bibitem[Dau]{Dau} I.~Daubechies:
\emph{An uncertainty principle for fermions with generalized kinetic energy.\/}
Commun. Math. Phys. 90(4):511--520  (1983).

\bibitem[EFS1]{EFS1}
L.~Erd\"os, S.~Fournais, J.~P.~Solovej,\emph{Scott correction for large atoms and molecules in a self-generated magnetic field\/}.
Commun. Math. Physics, 312(3):847--882 (2012).

\bibitem[EFS2]{EFS2}
L.~Erd\"os, S.~Fournais, J.~P.~Solovej, \emph{Relativistic Scott correction in self-generated magnetic fields\/}.
Journal of Mathematical Physics 53, 095202 (2012), 27pp. 

\bibitem[FLS]{FLS}
R.~L.~Frank, E.~H.~Lieb, R.~Seiringer. \emph{Hardy-Lieb-Thirring inequalities for fractional Schrödinger operators\/}. J. Amer. Math. Soc. 21(4), 925–-950 (2008).

\bibitem[FSW]{FSW}
R.~L.~Frank, H.~Siedentop, S.~Warzel. \emph{The ground state energy of heavy atoms: relativistic lowering of the leading energy correction\/}. Comm. Math. Phys. 278(2):549–-566 (2008).




\bibitem[GS]{GS}
G. M. Graf,  J. P Solovej. \emph{A correlation estimate with applications to quantum systems with Coulomb interactions\/}
Rev.~Math.~Phys., 6(5a):977--997 (1994).
Reprinted in The state of matter a volume dedicated to
E.~H.~Lieb, Advanced series in mathematical physics, 20,
M.~Aizenman and H.~Araki (Eds.), 142--166, World Scientific (1994).

\bibitem[IH]{Herbst}
I. W. Herbst. \emph{Spectral Theory of the operator
{$(p^2+m^2)^{1/2}-Ze^2/r$}}, Commun. Math. Phys. 53(3):285--294 (1977).

\bibitem[Ivr]{monsterbook}
V. Ivrii. \emph{Microlocal Analysis, Sharp Spectral  Asymptotics and Applications\/}.
\url{http://www.math.toronto.edu/ivrii/monsterbook.pdf}

\bibitem[Ivr2]{ivrii:rela}
V. Ivrii. \emph{Asymptotics of the ground state energy in the relativistic settings\/}.
\href{https://arxiv.org/1707.07014}{arxiv:math.1707.07014}

\bibitem[LLS]{LLS} E. H. Lieb, M. Loss, J. P. Solovej: 
\emph{Stability of Matter in Magnetic Fields\/}, Phys. Rev. Lett. 75:985--989 (1995).

\bibitem[LT]{LT}
E. H. Lieb, W. E. Thirring, \emph{Inequalities for the moments of the eigenvalues of the Schr\"odinger
Hamiltonian and their relation to Sobolev inequalities\/}, in Studies in Mathematical Physics (E. H. Lieb,
B. Simon, and A. S. Wightman, eds.), Princeton Univ. Press, Princeton, New Jersey, 1976, pp. 269--303.


\bibitem[LY]{Lieb-Yau}
E. H. Lieb,  H. T. Yau. \emph{The {S}tability and instability of relativistic matter\/}. Commun. Math. Phys. 118(2):  177--213 (1988).

\bibitem[SSS]{SSS}
J. P. Solovej, T. \O.  S{\o}rensen, W. L. Spitzer.
\emph{The relativistic Scott correction for atoms and molecules\/}.
 Comm. Pure Appl. Math., 63:39--118 (2010)Љ.
 
\end{thebibliography}
